\documentclass{amsart}

\usepackage{amsmath,amssymb, amscd, stmaryrd}
\usepackage{fullpage}
\usepackage{enumerate}
\usepackage{cite}
\usepackage{framed}
\usepackage[dvips,dvipdf]{graphicx}
\usepackage[usenames,dvipsnames]{color}
\usepackage{color}
\usepackage{tabularx}
\usepackage{array}
\usepackage{multirow}

\usepackage{lipsum}
\usepackage{amsfonts}
\usepackage{graphicx}
\usepackage{epstopdf}
\usepackage{subfigure,caption,float}
\usepackage{enumerate}
\usepackage{cleveref}

\newtheorem{theorem}{Theorem}[section]

\newtheorem{lemma}[theorem]{Lemma}

\theoremstyle{definition}
\newtheorem{definition}[theorem]{Definition}
\theoremstyle{definition}

\theoremstyle{remark}
\newtheorem{remark}[theorem]{Remark}

\ifpdf
  \DeclareGraphicsExtensions{.eps,.pdf,.png,.jpg}
\else
  \DeclareGraphicsExtensions{.eps}
\fi

\author[A. Anand]{Akash Anand} 
\address{Akash Anand, Department of
  Mathematics and Statistics, Indian Institute of Technology, Kanpur, UP 208016}
\email{akasha@iitk.ac.in}

\begin{document}

\title[A Fourier continuation framework]{A Fourier continuation framework \\ for high-order approximations %of smooth functions  \\
%-- basic implementation and numerical analysis
}

% REQUIRED
\begin{abstract}
It is well known that approximation of functions on $[0,1]$ whose periodic extension is not continuous fail to converge uniformly due to rapid Gibbs oscillations near the boundary. Among several approaches that have been proposed toward the resolution of Gibbs phenomenon in recent years, 
a Fourier continuation (FC) based approximation scheme has been suggested by Bruno and collaborators in the context of certain PDE solvers where approximation grids used are equispaced. While the practical efficacy of FC based schemes in obtaining a high-order numerical solution of PDEs is well known, theoretical convergence analyses largely remain unavailable. The primary objective of this paper is to take a step in this direction where we analyze the convergence rates of a Fourier continuation framework for approximations based on discrete functional data coming from equispaced grids. In this context, we explore a certain two-point Hermite interpolation strategy for constructing Fourier continuations that, not only simplifies the implementation of such approximations but also makes possible a rigorous analysis of its numerical properties. In particular, we show that the approximations converge with order $r+1$ for functions coming from a subspace of $C^{r,1}([0,1])$, the space of $r$-times continuously differentiable function whose $r$th derivative is Lipschitz continuous. We also demonstrate that theoretical rates are indeed achieved in practice, through a variety of numerical experiments.
\end{abstract}

\subjclass[2010]{65D15, 42A10, 41A25}
\keywords{Fourier continuation, high-order approximation, Gibb's phenomenon}

%\begin{keyword}
%Fourier continuation \sep high-order approximation \sep Gibb's phenomenon
%\end{keyword}

% REQUIRED
%\begin{AMS}
%  65D15, 65D30, 42A10
%\end{AMS}

\maketitle

\section{Introduction}

Given a function $f \in C^{r,1}([0,1])$, the space of $r$-times continuously differentiable real valued functions whose $r$th derivative is Lipschitz continuous, we seek to approximate it using a trigonometric polynomial
%, say $\mathcal{T}_{n,b}(f)$, 
of the form
\begin{align} \label{trigpoly}
%\mathcal{T}_{n,b}(f)(x) = 
\sum_{k=-n}^{n} c_k(f) e^{2\pi i k x/b}
\end{align}
with $c_{-k} = \overline{c_k}$ for some $b \ge 1$. 
%We denote the collection of all such trigonometric polynomials of degree $n$ or less by $\mathbb{T}_{n,b}$.
In particular, we are interested in approximations 
%\cref{trigpoly} 
%in $\mathbb{T}_{n,b}$ 
obtained as a truncated Fourier series, that we denote by $\mathcal{T}_{n,b}(f)$.
It is well known that if $f$ satisfies $f(0) = f(1)$, then the trigonometric polynomial obtained by truncating its Fourier series (with $b=1$) converges uniformly. In fact, if all its derivatives up to order $r$ satisfy 
\begin{align} \label{eq:per_conds}
f^{(\ell)}(0) = f^{(\ell)}(1) \text{\ for\ } 0 \le \ell \le r, 
\end{align}
then 
%the coefficients $c_k(f)$, given as
%\[
%c_k(f) = \int_0^1 f(x)e^{-2\pi i k x}\,dx
%\]
%asymptotically decay according to $|c_k(f)| = \mathcal{O}(|k|^{-(r+1)})$ %\cite{StienShakarchi}
%which, in turn, results in 
the approximation errors converge according to:
\[
\Vert f - \mathcal{T}_{n,1}(f) \Vert_{\infty,[0,1]} = \max_{x \in [0,1]} |f(x) - \mathcal{T}_{n,1}(f)(x)| = \mathcal{O}\left( \frac{\log n}{n^{r+1}} \right).
\]
If $ f \in C^{r}([0,1]) \cap C^{r+2}_{pw}([0,1])$, a subspace of $C^{r,1}([0,1])$, and additionally satisfeis \cref{eq:per_conds}, then the rate of convergence improves further to $\mathcal{O}(n^{-(r+1)})$. Note that $f \in C^{\ell}_{pw}([a,b])$ if and only if $f^{(\ell)} \in C_{pw}([a,b])$, the space of piecewise continuous functions on $[a,b]$. Recall that a function $g \in C_{pw}([a,b])$ if there are finitely many, say $n_d$, open disjoint intervals $(a_j,a_{j+1})$ with $a_0 = a$ and $a_{n_d} = b$, such that $g |_{(a_j,a_{j+1})}$ extends as a continuous function to $[a_j,a_{j+1}]$. For simplicity, we denote the space $C^{r}([a,b]) \cap C^{r+2}_{pw}([a,b])$ by $D^{r,1}([a,b])$ and its subspace where functions additionally satisfy 
\[
f^{(\ell)}(a) = f^{(\ell)}(b),\ \ \ 0 \le \ell \le r, 
\]
by $D^{r,1}_0([a,b])$.
%If $D^{r,1}([0,1])$ is the subspace of $C^{r,1}([0,1])$ containing functions whose $(r+1)$th derivative exists and is differentiable piecewise, and denote by $D_0^{r,1}([0,1])$ its subspace where the members additionally satisfy \cref{eq:per_conds}, the rate of convergence improves further to $\mathcal{O}(n^{-(r+1)})$.
%
% Note: Rademacher?s theorem (in 1 dimension due to Lebesgue) says that a Lipschitz 
% continuous function on (a,b) is differentiable almost everywhere on (a,b).
%
%with $f^{(r+1)}(0) = f^{(r+1)}(0)$. 
%
%and define as follows:
%\begin{align} \label{eq:subspace_def}
%X^{r}[a,b] = &\{ f \in C^{r,1}[a,b] \ :\ f^{(r+1)} \text{ exists but is only } \nonumber \\ 
%&\text{with countaband is piecewise Lipschitz continuous with} \nonumber \\ &f^{(\ell)}(0) = f^{(\ell)}(1) \text{\ for\ } 0 \le \ell \le r \}.
%\end{align}
%of functions that satisfy \cref{eq:per_conds} and whose $r$th derivative is piecewise differentiable with e. 
%Toward obtaining the convergence rate for approximations of $f \in D_0^{r,1}([0,1])$, we see that Fourier coefficients $c_k(f)$ for $k \ne 0$, 
% upon $(r+1)$-times integration by parts, are given by
%\[
%c_k(f) = \left( \frac{-1}{2\pi i k}\right)^{r+1} \int_0^1 f^{(r+1)}(x) e^{-2\pi i k x}\,dx
%\]
%The almost everywhere differentiability of $f^{(r+1)}$ then yields $|c_k(f)| = \mathcal{O}(|k|^{-(r+2)}).$ We, therefore, see that, for $f \in D_0^{r,1}([0,1])$,
%\begin{align} \label{eq:conv_rate}
%\Vert f - \mathcal{T}_{n,1}(f) \Vert_{\infty,[0,1]}  = \mathcal{O}\left( \frac{1}{n^{r+1}} \right).
%\end{align}
For example, $f_0(x) = |x-1/2|$ is in $D_0^{0,1}([0,1])$, $f_1(x) = (x-1/2)|x-1/2|$ is in $D^{1,1}([0,1])$ and $f_2(x) = (x-1/2)^2|x-1/2|$ is in $D^{2,1}([0,1])$. 
%For a similar discussion on the decay of Fourier coefficients, we refer readers to \cite{Johnson:aa}.

\begin{figure}[t]
  \centering
  \begin{subfigure}[$\mathcal{T}_{16,1}(f)$]
    {\includegraphics[width = 0.485\textwidth]{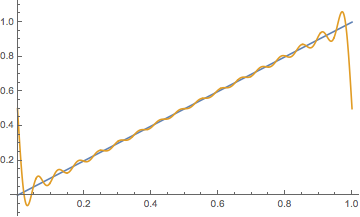}}
    %{\includegraphics[width = 0.485\textwidth]{Images/gibbs_16}}
  \end{subfigure}%
  ~
  \begin{subfigure}[$\mathcal{T}_{32,1}(f)$]
    {\includegraphics[width = 0.485\textwidth]{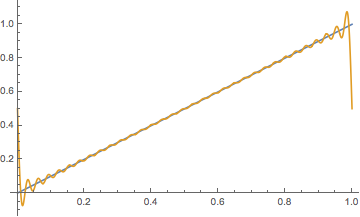}}
    %{\includegraphics[width = 0.485\textwidth]{Images/gibbs_32}}
  \end{subfigure}%
  \caption{
Approximations $\mathcal{T}_{n,1}(f)$ of $f(x) = x$ obtained as truncated Fourier series exhibit Gibb's phenomenon.
  }
  \label{fig:gibbs}
\end{figure}

The Fourier series approximations, on the other hand, fail to converge uniformly when $f(0) \ne f(1)$ due to rapid oscillations near boundary, known as the {\em Gibb's phenomenon} \cite{Wilbraham:1848aa,Gibbs:1898aa,Gibbs:1899aa,Hewitt:1979aa} (see \cref{fig:gibbs} for an example) --- development of effective strategies for its alleviation has remained a subject matter of much ongoing research.

Several approximation approaches have been proposed to overcome the difficulty of Gibb's oscillations. These include schemes that utilize Fourier or physical space filters \cite{Gottlieb:1997aa} as well as those that project the partial Fourier sums onto suitable functional spaces. For example, the Gegenbauer projection technique \cite{Gottlieb:1992aa, Gottlieb:1994aa, Gottlieb:1996aa, Gottlieb:1995aa, Gottlieb:1995ab, Gottlieb:1997aa} utilizes a space spanned by Gegenbauer polynomials. In Fourier-Pad\'{e} approximations, partial Fourier sums are approximated by rational trigonometric functions \cite{Geer:1995aa,Geer:1997aa,Driscoll:2001aa}. Techniques based on extrapolation algorithms \cite{Brezinski:2004aa} have also been used.
%Several approximation approaches have been proposed to overcome the difficulty of Gibb's phenomenon ... but remains a challenge, particularly, when the discrete functional data is given on a {\em uniformly spaced grid}. Expand this paragraph and include literature review here ...
Several Fourier continuation (or extension) approaches have also been proposed that seek to find a trigonometric polynomial of the form (\ref{trigpoly}) with $b > 1$. Such schemes rely on smoothly continuing $f$ on $[0,1]$ to $f_c$ on $[0,b]$ or $[1-b,1]$ for a suitable choice of $b$ in such a way that $f_c \equiv f$ on $[0,1]$ and $f^{(\ell)}(0) = f^{(\ell)}(b)$ for all integers $0 \le \ell \le r$ for some $r > 0$. 
Once such an $f_c$ has been produced, the restriction of its truncated Fourier series to $[0,1]$ serves as an approximation to $f$ (see Figure \ref{fig:fc_approx}). 
%However, given the multiple possibilities that exist for such continuations, the construction of $f_c$ with desired properties presents a challenge. 
For some examples where Fourier extension ideas have been used and discussed in various contexts, see \cite{Averbuch:1997aa, Garbey:1998aa, Garbey:2000aa,Boyd:2002aa, Potts:2004aa, Huybrechs:2010aa}. More recently,
an algorithmic construction for Fourier continuation has been suggested by Bruno and Lyon in \cite{Bruno:2010aa} in the context of certain PDE solvers where approximation grids used are equispaced. While the efficacy of this approach has been established through its application in various partial differential equation solvers (for example, see \cite{Lyon:2010aa,Albin:2011aa, Amlani:2016aa}), owing to the algorithmic nature of this scheme, numerical analysis of these methods, to a large extent, is intractable and consequently, theoretical convergence rates remain unavailable.

%More recently, a Fourier continuation approach has been proposed that seeks to find a trigonometric polynomial of the form \cref{trigpoly} with $b > 1$. Such schemes rely on extending $f$ on $[0,1]$ to $f_c$ on $[0,b]$ for a suitable choice of $b$ in such a way that $f_c$ is smooth, $f_c \equiv f$ on $[0,1]$ and $f^{(\ell)}(0) = f^{(\ell)}(b)$ for all integers $0 \le \ell \le r$ for some $r > 0$. Once such an $f_c$ has been produced, the restriction of its Fourier series approximation to $[0,1]$ also serves as an approximation to $f$. However, given the multiple possibilities that exist for such continuations, the construction the $f_c$ with desired properties remains a challenge. One such algorithmic construction has been suggested by Bruno and Mark {\color{red} cite them}. This approach {\color{red} explain}. While the efficacy of this approach has been established through its application in various partial differential equation solvers, owing to the algorithmic nature of this scheme, very limited analysis has been possible and consequently, theoretical convergence rates remain unavailable.

The primary objective of this paper is to
analyze a Fourier continuation framework for eliminating Gibbs oscillations from approximations based on discrete functional data coming from equispaced grids. Indeed, a Fourier continuation based high-order approximation becomes most relevant when the underlying grid is uniform, allowing for efficient calculations using FFT. 
In this context, we explore the construction based on the Hermite interpolation, that not only simplifies implementation of the approximation but also makes possible a rigorous analysis of its numerical properties. Toward this, in \cref{sec:framework}, we discuss the general framework for constructing such continuations where we also present theoretical convergence rates for approximations of functions in $D^{r,1}([0,1])$ by corresponding truncated Fourier series arising out of the continuation framework. In \cref{sec:polycont}, we review the Hermite polynomial based Fourier continuation approach and show that this construction indeed falls within the general framework of \cref{sec:framework}. We thus conclude that the theoretical convergence rates obtained remain valid in the context of Hermite polynomial based scheme. We then numerically verify this, in the discrete setting, in \cref{sec:numexp}, through a variety of computational experiments.

\section{A framework for Fourier continuation analysis} \label{sec:framework}

\begin{figure}[t]
  \centering
  \begin{subfigure}[$r=1$]
    {\includegraphics[width = 0.485\textwidth]{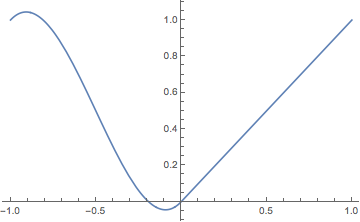}}
    %{\includegraphics[width = 0.485\textwidth]{Images/c_r_1}}
  \end{subfigure}%
  ~
  \begin{subfigure}[$r = 2$]
    {\includegraphics[width = 0.485\textwidth]{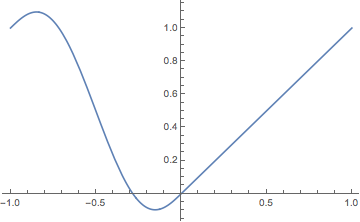}}
    %{\includegraphics[width = 0.485\textwidth]{Images/c_r_2}}
  \end{subfigure}%
    \begin{subfigure}[$r = 15$]
    {\includegraphics[width = 0.485\textwidth]{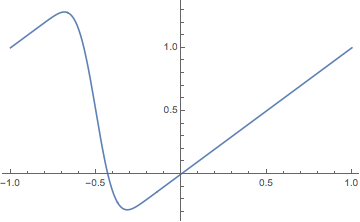}}
    %{\includegraphics[width = 0.485\textwidth]{Images/c_r_15}}
  \end{subfigure}%
  ~
  \begin{subfigure}[$r=100$]
    {\includegraphics[width = 0.485\textwidth]{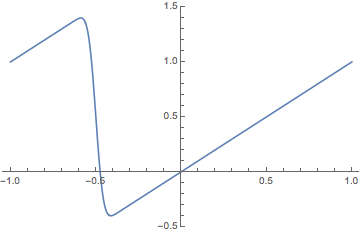}}
    %{\includegraphics[width = 0.485\textwidth]{Images/c_r_100}}
  \end{subfigure}%
  \caption{
Continuation $f_c(x)$ of $f(x) = x$ on $[0,1]$ to $[-1,1]$ using a polynomial of degree $2r+1$ so that all derivatives up to order $r$ are continuous and $f_c^{(\ell)}(-1) = f_c^{(\ell)}(1)$ for $0 \le \ell \le r$.
  }
  \label{fig:continuation}
\end{figure}

As described above, the Fourier continuation framework for approximation of a function $f \in D^{r,1}([0,1])$ can be viewed as a two step procedure, namely,
\begin{enumerate}
\item {\bf continuation:}
for a given $b > 1$, construct a function $f_c : [1-b,1] \to \mathbb{R}$ such that the following conditions hold:
\begin{align}
& f_c(x) = f(x), \ \text{ for all } x \in [0,1],  \label{fc_1}\\
& f_c(x) \in D_0^{r,1}([1-b,1]).   \label{fc_2}
%& f_c^{\ell}(0)  \text{ exists for all } 0 \le \ell \le r \text{ and }  \label{fc_3} \\
%& f_c^{(\ell)}(1) = f_c^{(\ell)}(1-b), \ \ \text{ for all } 0 \le \ell \le r.  \label{fc_4}
\end{align}
%Clearly, functions satisfying \cref{fc_1}, \cref{fc_2}, \cref{fc_3} and \cref{fc_4} belongs to $D_0^{r,1}([1-b,1])$. 
We illustrate this step in \cref{fig:continuation} for $b =2$ where we continue $f(x) = x$ on $[0,1]$ to the interval $[-1,1]$ with varying degree of smoothness as controlled by $r$. The explicit construction used in these examples are discussed in \cref{sec:polycont}.
\item {\bf Fourier approximation:}
for an $n \in \mathbb{N}$ and $x \in [0,1]$,
\[
f(x) \approx \mathcal{T}_{n,b}(f)(x) =  \sum_{k=-n}^n c_k(f_c) e^{2\pi i k x/b},
\]
where
\[
c_k(f_c) = \frac{1}{b} \int_{1-b}^1 f_c(x) e^{-2\pi i k x/b}\,dx.
\]
\end{enumerate}
In \cref{fig:fc_approx}, we show three Fourier approximations that correspond to $n = 2, 4$ and $8$ to a continuation of $f(x)=x$ shown in \cref{fig:continuation} where the continued function is in $C^{15,1}([-1,1])$ (in fact, in $D_0^{15,1}([-1,1])$). Note the absence of Gibbs oscillations in these approximations in contrast to those shown in \cref{fig:gibbs}.
\begin{figure}[t]
  \centering
  \begin{subfigure}[$n = 2$]
    {\includegraphics[width = 0.315\textwidth]{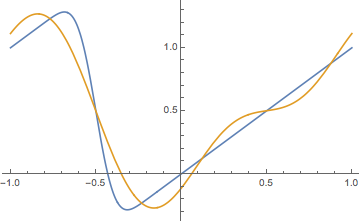}}
    %{\includegraphics[width = 0.315\textwidth]{Images/fc_r_15_n_2_2}}
  \end{subfigure}%
  ~
    \begin{subfigure}[$n = 4$]
    {\includegraphics[width = 0.315\textwidth]{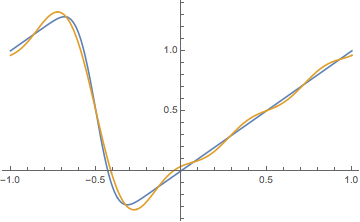}}
    %{\includegraphics[width = 0.315\textwidth]{Images/fc_r_15_n_4_2}}
  \end{subfigure}%
 ~
  \begin{subfigure}[$n=8$]
    {\includegraphics[width = 0.315\textwidth]{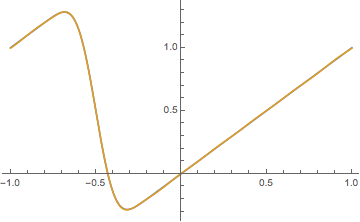}}
    %{\includegraphics[width = 0.315\textwidth]{Images/fc_r_15_n_8_2}}
  \end{subfigure}%
 \begin{subfigure}[$n = 2$]
    {\includegraphics[width = 0.315\textwidth]{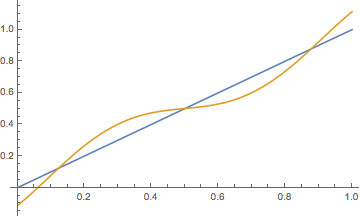}}
    %{\includegraphics[width = 0.315\textwidth]{Images/fc_r_15_n_2_1}}
  \end{subfigure}%
  ~
    \begin{subfigure}[$n=4$]
    {\includegraphics[width = 0.315\textwidth]{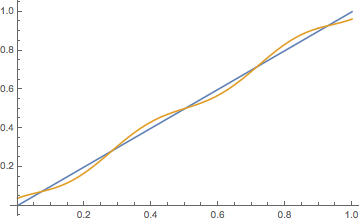}}
    %{\includegraphics[width = 0.315\textwidth]{Images/fc_r_15_n_4_1}}
  \end{subfigure}%
  ~
  \begin{subfigure}[$n=8$]
    {\includegraphics[width = 0.315\textwidth]{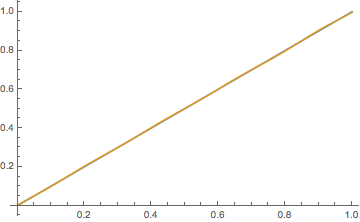}}
    %{\includegraphics[width = 0.315\textwidth]{Images/fc_r_15_n_8_1}}
  \end{subfigure}%
  \caption{
Approximation to $f(x) = x$ on $[0,1]$ by truncated Fourier series of a continued function in $D_0^{15,1}([-1,1])$. The top row shows the approximations on $[-1,1]$ whereas the bottom rows depict the corresponding approximation on $[0,1]$.}
  \label{fig:fc_approx}
\end{figure}

We analyze the approximation properties of the above strategy under the assumption that
$f_c$ is of the form
\begin{equation} \label{fc_gen}
f_c(x) = 
\begin{cases}
f(x), & x \in [0,1], \\
\mathcal{L}_r(F)(x), & x \in [1-b,0),
\end{cases}
\end{equation}
with 
\begin{equation}
F  = 
\begin{bmatrix} \label{f_matrix}
f(0) & f^{(1)}(0) & \cdots & f^{(\ell)}(0) & \cdots & f^{(r)}(0)  \\
f(1) & f^{(1)}(1) & \cdots & f^{(\ell)}(1) & \cdots & f^{(r)}(1)
\end{bmatrix}
\end{equation}
and a bounded linear operator $\mathcal{L}_r : M_{2,r+1}(\mathbb{R}) \to D^{r,1}([1-b,0])$, where $M_{m,n}(\mathbb{R})$ denotes the normed linear space of all $m \times n$ real matrices with the norm given by
\begin{equation}
%\Vert A \Vert_{max} \max_{1\le j \le 2, 1 \le k \le (r+1)} |a_{jk}|.
\Vert A \Vert_{max} = \max_{jk} |a_{jk}|.
\end{equation}
%Thus, for $A \in M_{2,r+1}$, the function $\mathcal{L}_r(A)$ is an $r$-times continuously differentiable function on $[1-b,0]$ where the $r$th derivative is Lipschitz continuous. 
Moreover, the linear operator $\mathcal{L}_r$ is required to satisfy the {\em derivative conditions} that read
\begin{align} 
(\mathcal{L}_r(F))^{(\ell)}(0) &= f^{(\ell)}(0) \label{deriv_cond_1} \\
(\mathcal{L}_r(F))^{(\ell)}(1-b) &= f^{(\ell)}(1) \label{deriv_cond_2}
\end{align}
for $0 \le \ell \le r$. 
%For $r_1 \le r_2$, if the boundary data operator $\mathcal{F}_{r_1,r_2} : D^{r_2,1}([1-b,0]) \to M_{2,r_2-r_1+1}(\mathbb{R})$ given by
%\bq \label{bdry_matrix}
%\mathcal{F}_{r_1,r_2}(g) = 
%\begin{bmatrix} 
%g^{(r_1)}(0) & \cdots & g^{(\ell)}(0) & \cdots & g^{(r_2)}(0)  \\
%g^{(r_1)}(1) & \cdots & g^{(\ell)}(1) & \cdots & g^{(r_2)}(1)
%\end{bmatrix},
%\eq
%then, for $0 \le \ell \le r_2 - r_1$, we obviously have,
%\[
%\mathcal{F}_{r_1,r_2-\ell}(g^{(\ell)}) = \mathcal{F}_{r_1+\ell,r_2}(g).
%\]
%For simplicity of notation, we denote $\mathcal{F}_{0,r}$ by $\mathcal{F}_{r}$.
%the derivative conditions \cref{deriv_cond_1}-\cref{deriv_cond_2} 
%%in conjunction with \cref{L_deriv} imply that 
%result in
%\[
%\mathcal{F}_{r_1,r_2}(\mathcal{L}_{r_2-r_1}(\mathcal{F}_{r_1,r_2}(f)))
%%(\mathcal{L}_r(\mathcal{F}_r(f)))^{(\ell)}(x) = \frac{d}{dx}\mathcal{L}_r^{(k-1)}(F)(x)
%\]
Note that $\mathcal{L}_r(F)(x)$ assumes the form
\begin{align} \label{L_expanded}
\mathcal{L}_r(F)(x) = \sum_{m=0}^r L_m^0(x) f^{(m)}(0) + \sum_{m=0}^r L_m^1(x) f^{(m)}(1)
\end{align}
for some functions $L_{m}^0, L_{m}^1 \in C^{r,1}([1-b,0])$, where the derivative conditions \cref{deriv_cond_1}-\cref{deriv_cond_2} require that they satisfy
\begin{align}
(L_{m}^0)^{(\ell)}(0) = \delta_{\ell m}, (L_{m}^0)^{(\ell)}(1-b) = 0, \label{deriv_cond_11} \\
(L_{m}^1)^{(\ell)}(0) = 0, (L_{m}^1)^{(\ell)}(1-b) = \delta_{\ell m} \label{deriv_cond_22},
\end{align}
for $0 \le \ell \le r$. 

\begin{remark}
While in our discussions, we extend $f$ to the left of the interval $[0,1]$, that is, to the interval $[1-b,1]$, a similar right continuation, that is, to the interval $[0,b]$ also works analogously.
\end{remark}

\begin{lemma} \label{lemma-thm}
If $g \in D^{r,1}_0([1-b,1])$, then, for all $n > 0$, we have
\[
\Vert \mathcal{T}_{n,b}(g)-g \Vert_{\infty,[0,1]} \le \frac{C}{n^{r+1}},
\]
for a positive constant $C$ independent of $n$.
\end{lemma}

\begin{proof}
Note that the Fourier coefficients $c_k(g)$, for $k \ne 0$, upon $(r+1)$ integration by parts, are given by
\begin{align*}
c_k(g) %= \frac{1}{b}\int_{1-b}^1 f_c(x)e^{-2\pi i k x/b}\,dx 
%&= \frac{1}{b}\frac{b}{-2\pi i k}\left[ f_c(1)e^{-2\pi i k/b} - f_c(1-b)e^{-2\pi i k(1-b)/b} -  \int_{1-b}^1 f_c'(x)e^{-2\pi i k x/b}\,dx \right] \\
%&= \frac{1}{b}\frac{b}{-2\pi i k}\left[ f(1)e^{-2\pi i k/b} - f(1)e^{-2\pi i k/b} -  \int_{1-b}^1 f_c'(x)e^{-2\pi i k x/b}\,dx \right] \\
%&= \frac{1}{b}\frac{b}{2\pi i k}  \int_{1-b}^1 f_c'(x)e^{-2\pi i k x/b}\,dx \\
%&= \frac{1}{b}\frac{b}{2\pi i k}\frac{b}{-2\pi i k}\left[ f_c'(1)e^{-2\pi i k/b} - f_c'(1-b)e^{-2\pi i k(1-b)/b} -  \int_{1-b}^1 f_c''(x)e^{-2\pi i k x/b}\,dx \right] \\
%&= \frac{1}{b}\frac{b}{2\pi i k}\frac{b}{-2\pi i k}\left[ f'(1)e^{-2\pi i k/b} - f'(1)e^{-2\pi i k/b} -  \int_{1-b}^1 f_c''(x)e^{-2\pi i k x/b}\,dx \right] \\
%&= \frac{1}{b}\left(\frac{b}{2\pi i k}\right)^2  \int_{1-b}^1 f_c''(x)e^{-2\pi i k x/b}\,dx \\
= \frac{1}{b}\left(\frac{b}{2\pi i k}\right)^{r+1}  \int_{1-b}^1 g^{(r+1)}(x)e^{-2\pi i k x/b}\,dx.
\end{align*}
As $g^{(r+1)}$ is piecewise differentiable with finitely many jump discontinuties,  say at $a_j, \ j = 0, 1, \ldots, n_d$, with $a_j < a_{j+1}$, $a_0 = 1-b$ and $a_{n_d} = 1$, an application of integration by parts to each of these subintervals $[a_j,a_{j+1}]$ yields
\begin{align}
c_k(g) 
&= \frac{1}{b}\left(\frac{b}{2\pi i k}\right)^{r+2} \sum_{j=0}^{n_d} \left( g^{(r+1)}(a_{j}+) - g^{(r+1)}(a_j-)\right)e^{-2\pi i k a_j/b} \nonumber \\
%&= \frac{1}{b}\left(\frac{b}{2\pi i k}\right)^{(r+2)} \sum_{j=1}^{n_d} \left( f_c^{(r+1)}(a_{j-1}+)e^{-2\pi i k a_{j-1}/b} - f_c^{(r+1)}(a_j-)e^{-2\pi i k a_j/b}\right) \\
&+ \frac{1}{b}\left(\frac{b}{2\pi i k}\right)^{r+2}  \int_{1-b}^1 g^{(r+2)}(x)e^{-2\pi i k x/b}\,dx  \label{eq:fourcoefbound}
\end{align}
where $f(a\pm)$ denotes $\lim_{h\to +0}f(a \pm h)$. The left hand limit at $x=1-b$ and right hand limit at $x=1$, respectively, are obtained as $g^{(r+1)}(a_{n_0}-) = g^{(r+1)}(a_{n_d}-)$ and $g^{(r+1)}(a_{n_d}+) = g^{(r+1)}(a_{0}+)$ .
The result now follows from \cref{eq:fourcoefbound} and the following inequality:
\[
\Vert \mathcal{T}_{n,b}(g)-g \Vert_{\infty,[0,1]} \le \Vert \mathcal{T}_{n,b}(g)-g  \Vert_{\infty,[1-b,1]} \le  \sum_{|k| \ge n+1} \left|  c_k(g)\right|.
\] 
\end{proof}

Clearly, as constructed in \cref{fc_gen}, the continuation $f_c \in D_0^{r,1}([1-b,1])$, and therefore, its truncated Fourier series converges according to the rate obtained in \cref{lemma-thm}.
This construction, of course, assumes that the boundary data matrix $F$ is available exactly, as might be the case in many applications. However, in many other cases, especially when approximations are being constructed from a discrete functional data, complete boundary information may not available explicitly and are obtained indirectly using numerical approximations. Consequently, the matrix $F$ used in the continuation process may be inexact, which in turn, introduces additional inaccuracies in the Fourier continuation approximations. To study the effect of inexact data matrix on errors, we begin by denoting the approximate continuation
\begin{equation} \label{fc_gen_approx}
\hat{f}_c(x) = 
\begin{cases}
f(x), & x \in [0,1], \\
\mathcal{L}_r(\hat{F})(x), & x \in [1-b,0),
\end{cases}
\end{equation}
corresponding to $\hat{F} \in M_{2,r+1}(\mathbb{R})$, $\hat{F} \ne F$.
We note that $\hat{f}_c$ as defined in \cref{fc_gen_approx}, typically, is not in $C^{r,1}([1-b,1])$ and, in fact, can by discontinuous if 
the first column of $\hat{F}$ differs from that of $F$.
%, an input matrix to $\mathcal{L}_r$ in \cref{fc_gen_approx}, 
%happens to be different from the first column of and $F$ . 
Obviously, the exact knowledge of one or more columns of $F$ has favorable impact on the regularity of $\hat{f}_c$. For example,
in a typical discrete setting, while the inexact derivative calculations result in $\hat{F}$ to carry numerical error, availability of boundary data $f(0)$ and $f(1)$ makes it possible to choose $\hat{F}$ so that its first column matches exactly with that of $F$ thus making $\hat{f}_c$ continuous at $x = 0$ and satisfy $\hat{f}_c(1-b) = \hat{f}_c(1)$. 
To formalize this idea of having such partially known boundary data and study its impact on the approximation accuracy, we introduce the following definition. 
\begin{definition}
A boundary data matrix $\hat{F} = (\hat{f}_{jk})_{0 \le j \le 1,0 \le k \le r}$ is $r$-exact (or simply exact) with respect to $f$ if $\hat{f}_{jk} = f^{(k)}(j), j \in \{0,1\},0 \le k \le r$; for $0 \le s < r$,
a matrix $\hat{F}$ is $s$-exact with respect to $f$ if the first $(s+1)$~columns of $\hat{F}$ agrees exactly with those of $F$ but they differ in $(s+2)$th column, that is,
\[
\hat{f}_{jk} = f^{(k)}(j), j \in \{0,1\},0 \le k \le s \text{ and } \hat{f}_{j(s+1)} \ne f^{(s+1)}(j) \text{ for } j = 0 \text{ or } j=1.
\]
\end{definition}
It is straightforward to see that if $\hat{F}$ is $s$-exact with respect to $f$ then for $\ell = 0, \ldots, s$, $\hat{f}_c^{(\ell)}$ is continuous at $x = 0$ and $\hat{f}_c^{(\ell)}(1-b) = \hat{f}_c^{(\ell)}(1)$. %We formally state this useful observation as a Lemma.

\begin{lemma} \label{approx_regularity}
Let $f \in C^{r,1}([0,1])$ and $f_c$ be its continuation as given in \cref{fc_gen}. If the approximate continuation $\hat{f}_c$ given in \cref{fc_gen_approx} corresponds to
$\hat{F} \in M_{2,r+1}(\mathbb{R})$ that is $s$-exact with respect to $f$, $0 \le s \le r$, then, $\hat{f}_c \in C_0^{s,1}([1-b,1])$.
%and satisfies $\hat{f}_c^{(\ell)}(1-b) = \hat{f}_c^{(\ell)}(1)$ for $\ell = 0, \ldots, s$.
\end{lemma}

\begin{lemma} \label{approx_lemma}
Let $f \in C^{r,1}[0,1]$, $f_c$ and $\hat{f}_c$ be as given in \cref{fc_gen} and \cref{fc_gen_approx} respectively where $\hat{F} \in M_{2,r+1}(\mathbb{R})$ is $s$-exact with respect to $f$, $0 \le s < r$. 
Then, there is a constant $C > 0$ independent of $n$, such that
\[
\Vert \mathcal{T}_{n,b}(\hat{f}_c) - \mathcal{T}_{n,b}(f_c) \Vert_{\infty,[0,1]}
 \le C n^{-(s+1)} \Vert \hat{F}-F \Vert_{max}.
\]
\end{lemma}
\begin{proof}
Clearly,
\begin{align}
&\Vert \mathcal{T}_{n,b}(\hat{f}_c) - \mathcal{T}_{n,b}(f_c) \Vert_{\infty,[0,1]} = \Vert \mathcal{T}_{n,b}(\hat{f}_c - f_c) - (\hat{f}_c - f_c)  \Vert_{\infty,[0,1]} \le \nonumber \\ 
& \Vert \mathcal{T}_{n,b}(\hat{f}_c - f_c) - (\hat{f}_c - f_c)  \Vert_{\infty,[1-b,1]} \le \sum_{|k| \ge n+1} |c_k(\hat{f}_c - f_c)|. \label{eq:approx_lemma_1}
\end{align}
%From the following expressing
%\begin{align*}
%c_k(\hat{f}_c - f_c) &= \frac{1}{b}\sum_{m=s+1}^r \left( \hat{f}_{0m} - f^{(m)}(0) \right) \int_{1-b}^0 L_m^0(x) e^{-2\pi i k x/b}\,dx \\ 
%&+ \frac{1}{b} \sum_{m=s+1}^r \left( \hat{f}_{1m} - f^{(m)}(1) \right) \int_{1-b}^0 L_m^1(x) e^{-2\pi i k x/b}\,dx
%\end{align*}
Now, from the derivative conditions \cref{deriv_cond_1}-\cref{deriv_cond_2}, and $(s+1)$ times integrate by parts, we get
\begin{align*}
c_k(\hat{f}_c - f_c) &= \frac{1}{b} \left(\frac{b}{2\pi i k}\right)^{s+1} \sum_{m=s+1}^r \left( \hat{f}_{0m} - f^{(m)}(0) \right) \int_{1-b}^0 \left(L_m^0 \right)^{(s+1)}(x) e^{-2\pi i k x/b}\,dx \\ 
&+ \frac{1}{b} \left(\frac{b}{2\pi i k}\right)^{s+1} \sum_{m=s+1}^r \left( \hat{f}_{1m} - f^{(m)}(1) \right) \int_{1-b}^0 \left(L_m^1\right)^{(s+1)}(x) e^{-2\pi i k x/b}\,dx.
\end{align*}
Now, using \cref{eq:approx_lemma_1} and the fact that $\left(L_m^0\right)^{(s+1)}$ and $\left(L_m^1\right)^{(s+1)}$ are Lipschitz continuous, we conclude the result.
\end{proof}

The next result shows that 
%in spite of the general lack of smoothness in $\hat{f}_c$, 
the Fourier continuation approximations converge rapidly and the rate of convergence is tied only to the smoothness of $f$ and the order of accuracy in the derivative approximations.

\begin{theorem} \label{thm}
%Let  $f:[0,1] \in D^{r,1}$ and $\hat{F} \in M_{2,r+1}(\mathbb{R})$ be such that its first column is same as that of the first column of the matrix $F$. Then there exists an integer $n_0 > 0$ and positive constants $C_1, C_2, C_3$ such that
Let  $f \in D^{r,1}([0,1])$, $f_c$ and $\hat{f}_c$ be as given in \cref{fc_gen} and \cref{fc_gen_approx} respectively where $\hat{F} \in M_{2,r+1}(\mathbb{R})$ is $s$-exact with respect to $f$, $0 \le s < r$ . Then there exist positive constants $C_1$ and $C_2$ such that
\[
\Vert \mathcal{T}_{n,b}(\hat{f}_c) - f \Vert_{\infty,[0,1]} \le \frac{C_1}{n^{s+1}} \Vert \hat{F}-F \Vert_{max}   
+ \frac{C_2}{n^{r+1}}
\]
for all $n \ge 1$. In particular, if $\Vert \hat{F}-F \Vert_{\max} = \mathcal{O}(1/n^p),\ p+s  > r$, then
\[
\Vert \mathcal{T}_{n,b}(\hat{f}_c) - f \Vert_{\infty,[0,1]} = \mathcal{O}\left( \frac{1}{n^{r+1}} \right)
\]
whereas, if  $p+s \le r$, then %$\Vert A-F \Vert_{\max} = \mathcal{O}(1/n^r)$
\[
\Vert \mathcal{T}_{n,b}(\hat{f}_c) - f \Vert_{\infty,[0,1]} = \mathcal{O}\left( \frac{1}{n^{p+s+1}} \right)
\]
\end{theorem}
\begin{proof}
The result follows from \cref{lemma-thm}, %\cref{approx_regularity}, 
\cref{approx_lemma} and the triangle inequality
\begin{align*}
\Vert \mathcal{T}_{n,b}(\hat{f}_c) - f \Vert_{\infty,[0,1]} & \le \Vert \mathcal{T}_{n,b}(\hat{f}_c) -  \mathcal{T}_{n,b}(f_c) \Vert_{\infty,[0,1]} + \Vert  \mathcal{T}_{n,b}(f_c) - f \Vert_{\infty,[0,1]}.
%& \le  (3 + \log n ) \Vert \hat{f}_c - f_c \Vert_{\infty,[1-b,1]} + \Vert  \mathcal{T}_{n,b}(f_c) - f \Vert_{\infty,[0,1]}. \\
%& \le  \Vert \mathcal{T}_{n,b} \Vert  \Vert \mathcal{L}_{r} \Vert \epsilon ~+~ \frac{M}{n^r}. 
\end{align*}
\end{proof}

%\begin{example}
%Consider $f(x) = \sin \pi x$ on $[0,1]$. Let 
%\begin{align*}
%%\mathcal{L}_r[F] = \sum_{m=0}^r \frac{1}{2}\left( f^{(r)}(0) + f^{(r)}(1)\right)
%\mathcal{L}_1(\hat{F})(x) &= \mathcal{L}_1\left(\begin{bmatrix} f_{00} & f_{01} \\ f_{10} & f_{11} \end{bmatrix}\right)(x) =  \frac{1}{2}\left( f_{00} + f_{10}\right) \\ 
%&+  \frac{1}{2}\left( f_{00} - f_{10}\right)\cos \pi x + \frac{1}{2\pi}\left( f_{01} - f_{11}\right)\sin \pi x + \frac{1}{4\pi}\left( f_{01} + f_{11}\right)\sin 2\pi x
%\end{align*}
%Clearly, $\mathcal{L}_1$ is linear in $\hat{F}$ and satisfies the derivative conditions \cref{deriv_cond_1}-\cref{deriv_cond_2}.
%Thus, $\mathcal{T}_{n,2}(\hat{f}_c) = \sum_{k=-n}^n c_k(\hat{f}_c) e^{\pi i k x}$, where
%\begin{align*}
%c_k(\hat{f}_c) &= \frac{1}{2}\int_0^1 \sin \pi x\ e^{\pi ikx}\,dx + \frac{1}{2}\int_{-1}^0  \mathcal{L}_1(\hat{F})(x)\ e^{\pi ikx}\,dx \\
%& = -\frac{1+(-1)^k}{2\pi (k^2-1)} +  \frac{(1-(-1)^k)i}{2\pi k}\frac{ f_{00} + f_{10}}{2} + \frac{(1+(-1)^k)ik}{2\pi(k^2-1)}\frac{f_{00} - f_{10}}{2} \\ 
%&+ \frac{(1+(-1)^k)}{2\pi(k^2-1)}\frac{ f_{01} - f_{11}}{2\pi}
%\end{align*}
%For $f_{ij} = f^{(j)}(i) + a_{ij}h^p$, we have $\Vert \hat{F} - F \Vert_{\max} = h^p \max_{i,j}{|a_{ij}|}$.
%\end{example}

In the next section, we review a well known explicit Fourier continuation strategy based on two point Hermite interpolation that falls within the framework we discussed here. 

\section{The construction based two point Hermite interpolation} \label{sec:polycont}

In this section, we investigate the Fourier extension strategy using two point Hermite interpolation that has been used before in various contexts (for examples, see \cite{Garbey:1998aa,Garbey:2000aa, Potts:2004aa, Potts:2017aa}).
While one could work with any $b > 1$ for the construction,
we restrict our presentation to the choice $b=2$.
For a matrix $\hat{F} = [\hat{f}_{jk}]_{0 \le j \le 1,0 \le k \le {r}} \in M_{2,r+1}(\mathbb{R})$, we introduce the polynomial $\mathcal{P}_r(\hat{F})(x)$ of degree $2r+1$ given by
\begin{align}
\mathcal{P}_r(\hat{F})(x) &= (1+x)^{r+1} \sum_{m=0}^r \frac{\hat{f}_{0m}}{m!}\sum_{n = 0}^{r-m} (-1)^{n} \binom{r+n}{n} x^{m+n}
 \nonumber \\ &+ (-x)^{r+1} \sum_{m=0}^r \frac{\hat{f}_{1m}}{m!}\sum_{n = 0}^{r-m}  \binom{r+n}{n} (1+x)^{m+n}.
 \label{poly}
\end{align}

%\beq
%g_r(x) &=& (x+1)^{r+1} \sum_{m=0}^r \frac{f^{(m)}(0)}{m!}\sum_{\ell = 0}^{r-m} (-1)^{\ell} \binom{r+\ell}{\ell} x^{m+\ell}
% \nonumber \\ &+& (-x)^{r+1} \sum_{m=0}^r \frac{f^{(m)}(1)}{m!}\sum_{\ell = 0}^{r-m}  \binom{r+\ell}{\ell} (x+1)^{m+\ell},
%\eeq

%The following result is useful in establishing that $\mathcal{P}_r(F)$ satisfies derivative conditions (\cref{deriv_cond_1})-(\cref{deriv_cond_2}).

We note that $\mathcal{P}_r(\hat{F})(x)$ can be expressed as
\[
\mathcal{P}_r(\hat{F})(x) = \sum_{m=0}^r \hat{f}_{0m} P_m^0(x) + \sum_{m=0}^r \hat{f}_{1m} P_m^1(x)
\]
with
\[
P_m^0(x) = \frac{1}{m!} x^{m}(1+x)^{r+1}  \sum_{n = 0}^{r-m} (-x)^{n} \binom{r+n}{n}
\]
and
\[
P_m^1(x) = \frac{1}{m!}  (1+x)^{m}(-x)^{r+1}  \sum_{n = 0}^{r-m} (1+x)^{n} \binom{r+n}{n}
\]
Before we show that the $\mathcal{P}_r$ defined above indeed is a bounded linear operator that satisfies derivative conditions \cref{deriv_cond_1}-\cref{deriv_cond_2}, we observe the following useful fact.

\begin{lemma}\label{binid:1}
Let $n$ be a positive integer and $r$ be an integer with $n \le r+1$. Then, we have
\begin{align} \label{binid:11}
\sum_{k = 0}^{n} (-1)^{k}  \binom{r+1}{k} \binom{r+n-k}{r} =  0.
\end{align}
\end{lemma}
\begin{proof}
The identity \cref{binid:11} follows from the observation that, for $m \ge 0$, we have
\[
\binom{r+m}{r} = \frac{(-1)^{m}}{m!} \left. \frac{d^m}{dx^m} \frac{1}{(1+x)^{r+1}}\right|_{x=0}
\]
and that, for $0 \le k \le r+1$,
\[
\binom{r+1}{k} = \frac{1}{k!} \left. \frac{d^k}{dx^k} (1+x)^{r+1} \right|_{x=0}.
\]
Indeed,
\[
0 =  \frac{d^n}{dx^n} \left[ (1+x)^{r+1}\frac{1}{(1+x)^{r+1}}\right] = \sum_{k=0}^n \binom{n}{k} \left[\frac{d^k}{dx^k} (1+x)^{r+1}\right] \left[\frac{d^{n-k}}{dx^{n-k}} \frac{1}{(1+x)^{r+1}}\right],
\] 
and the result follows.
\end{proof}

\begin{theorem} \label{thm_bdd_lin_op}
The operator $\mathcal{P}_r$ defines a bounded linear operator from $M_{2,r+1}$ to $C^{r,1}([-1,0])$ that also satisfies derivative conditions \cref{deriv_cond_1}-\cref{deriv_cond_2}.
\end{theorem}
\begin{proof}
The linearity of $\mathcal{P}_r$ is obvious from the definition. Now, for $-1 \le x \le 0$,
\begin{align*} 
|\mathcal{P}_r(\hat{F})(x)| &\le  \sum_{m=0}^r \frac{|\hat{f}_{0m}|}{m!}\sum_{n = 0}^{r-m} \binom{r+n}{n} + \sum_{m=0}^r \frac{|\hat{f}_{1m}|}{m!}\sum_{n = 0}^{r-m}  \binom{r+n}{n} \\
&\le \Vert \hat{F} \Vert_{max} \left( \sum_{m=0}^r \frac{1}{m!} \binom{2r+1-m}{r-m} + \sum_{m=0}^r \frac{1}{m!} \binom{2r+1-m}{r-m}  \right) \\ & \le  2 \binom{2r+2}{r} \Vert \hat{F} \Vert_{max}.
\end{align*}
The boundedness of $\mathcal{P}_r$ thus follows with $\Vert \mathcal{P}_r \Vert \le 2\binom{2r+2}{r}$ where $\Vert \cdot \Vert$ is the induced operator norm.
Now,
\begin{align}
&(P_m^0)^{(\ell)}(x) = \nonumber \\
 &\frac{\ell !}{m!}  \sum_{k=0}^{\ell}  \binom{r+1}{k}  (1+x)^{r+1-k}  \sum_{n = (\ell-k-m)_+}^{r-m} (-1)^{n} \binom{r+n}{n}  \binom{m+n}{\ell-k} x^{n-(\ell-k-m)} \label{L0Deriv}
\end{align}
and
\begin{align}
&(P_m^1)^{(\ell)}(x) = \nonumber \\
 &(-1)^{r+1}\frac{\ell !}{m!}  \sum_{k=0}^{\ell}  \binom{r+1}{k} x^{r+1-k}  \sum_{n = (\ell-k-m)_+}^{r-m}  \binom{r+n}{n}  \binom{m+n}{\ell-k} (1+x)^{n-(\ell-k-m)}, \label{L1Deriv}
\end{align}
where $(x)_+ = \max(0,x)$. It is clear from \cref{L0Deriv} that $(P_m^0)^{(\ell)}(-1) = 0$ for all $0 \le \ell \le r$. Also, if $\ell < m$, then $(P_m^0)^{(\ell)}(0) = 0$. For $\ell \ge m$, we have
\begin{align*}
(P_m^0)^{(\ell)}(0) &= \frac{\ell !}{m!}  \sum_{k=0}^{\ell-m} (-1)^{\ell-k-m} \binom{r+1}{k}  \binom{r+\ell-k-m}{\ell-k-m} .
\end{align*}
Clearly, $(P_m^0)^{(m)}(0) = 1$. For $\ell > m$, on the other hand, it follows from \cref{binid:1} that $(P_m^0)^{(\ell)}(0) = 0$. Thus, $(P_m^0)^{(\ell)}$ satisfy \cref{deriv_cond_11}.
Using \cref{L1Deriv}, one similarly sees that $(P_m^1)^{(\ell)}(0) = 0$ and $(P_m^1)^{(\ell)}(-1) = \delta_{\ell m}$, that is, $(P_m^1)^{(\ell)}$ satisfy \cref{deriv_cond_22}.
\end{proof}

In the light of Theorem \cref{thm_bdd_lin_op}, we utilize $\mathcal{P}_r(F)$ for a continuation of $f$ given by
\begin{equation} \label{fc}
f_c(x) = 
\begin{cases}
f(x), & x \in [0,1] \\
\mathcal{P}_r(F)(x), & x \in [-1,0).
\end{cases}
\end{equation}
The truncated Fourier series that, in this setting, reads
\begin{equation} \label{tfs}
\mathcal{T}_{n,2}(f)(x) = \sum_{k=-n}^{n} c_k(f_c) e^{\pi i k x}
\end{equation}
with
\begin{align} \label{fc_coeffs}
c_k(f_c) = \frac{1}{2}\int_{-1}^1 f_c(x)e^{-\pi i k x} \,dx.
\end{align}
is then used as an approximation to $f$ on $[0,1]$. 

\section{A discrete approximation problem and numerical examples} \label{sec:numexp}

We now consider the problem of constructing a Fourier continuation approximation for the case where discrete functional data is available on an equispaced grid on the interval $[0,1]$ that has relevance in many applications including certain PDE solvers. Such a grid of size $n+1$ has its $j$th grid point at $x_j = j/n$ where the corresponding function value $f(x_j)$ is assumed to be known and are denoted by $f_j$ for $j = 0, \ldots, n$. The functional data, in this case, is continued to the equispaced grid on $[-1,1]$ according to
%\begin{equation}
%\label{eq:fc_discrete}
%(f^d_c)_{j} = 
%\begin{cases}
%f_j, & j = 0, \ldots, n \\
%\mathcal{P}_r(F)(j/n), & j = -n, \ldots, -1,
%\end{cases}
%\end{equation}
%and
\begin{equation}
\label{eq:fc_discrete}
(\hat{f}^d_c)_{j} = 
\begin{cases}
f_j, & j = 0, \ldots, n \\
\mathcal{P}_r(F_{n,p})(j/n), & j = -n, \ldots, -1,
\end{cases}
\end{equation}
where the $0$-exact boundary data matrix $F_{n,p}$ is obtained as 
\[
f_{0,m} = \mathcal{D}_{n,p}^{m,+}(f)(x_0) \text{ and } f_{1,m} = \mathcal{D}_{n,p}^{m,-}(f)(x_n).
\]
for $1 \le m \le r$,
using forward and backward finite difference derivative operators $\mathcal{D}_{n,p}^{m,+}(f)$ and $\mathcal{D}_{n,p}^{m,-}(f)$ respectively of order of accuracy $p$ for approximations of $f^{(m)}$ whose generic form reads
\[
D_{n,p}^{m,\pm}(f)(x_\ell) = (\pm n)^m\left( \sum_{k=0}^{m+p-1} \left(a_{p}^m\right)_k f_{\ell\pm k} \right)
\]
for appropriately chosen constants $\left(a_{p}^m\right)_k$. 
%We recall, for examples, that constants for $4$th order accurate approximations to the first and second derivative are
%\begin{align*}
%&\left( a_{4}^1\right)_0 = -25/12, \left( a_{4}^1\right)_1 = 4, \left( a_{4}^1\right)_2=-3, \left( a_{4}^1\right)_3 = 4/3, \left( a_{4}^1\right)_4 =-1/4,\\
%& \left( a_{4}^2\right)_0 = 15/4, \left( a_{4}^2\right)_1 = -77/6, \left( a_{4}^2\right)_2 = 107/6, \left( a_{4}^2\right)_3 = -13, \\ & \left( a_{4}^2\right)_4 = 61/12,
%\left( a_{4}^2\right)_5 = -5/6,
%\end{align*}
%whereas the constants that yield $4$th order accurate approximations for the third derivative are
%\begin{align*}
%& \left( a_{4}^3\right)_0 = -49/8, \left( a_{4}^3\right)_1 = 29, \left( a_{4}^3\right)_2 = -461/8, \left( a_{4}^3\right)_3 = 62, \\
%& \left( a_{4}^3\right)_4 = -307/8, \left( a_{4}^3\right)_5 = 13, \left( a_{4}^3\right)_6 = -15/8.
%\end{align*}	 	 	 	 	
Now, using $\hat{f}^d_{c}$, as obtained in \cref{eq:fc_discrete}, we compute 
\begin{align} \label{fc_coeffs_a}
c^d_k(\hat{f}^d_c) = 
\frac{1}{2n}\sum_{j=-n}^{n-1} (\hat{f}^d_c)_{j} e^{-\pi i j k /n}, 
%\frac{1}{2n}\sum_{j=0}^{n} f_je^{-\pi i j k /n} + \frac{(-1)^k}{2n}\sum_{j=1}^{n-1} (\mathcal{P}_r(\hat{F}))(j/n-1)e^{-\pi i jk/n}.
\end{align}
for $k = -n,\ldots, n-1$,
%as an approximation to $c_k(\hat{f}^d_c)$ where the integral is numerically evaluated using the trapezoidal rule. 
to arrive at the interpolating Fourier continuation approximation for the discrete problem given by
\begin{align} \label{tfsa}
\mathcal{T}^d_{n,2}(\hat{f}^d_c)(x) = \sum_{k=-n}^{n-1} c^d_k(\hat{f}^d_c) e^{\pi i k x}.
\end{align}
Note that the coefficients $c_k^d(\hat{f}_c^d)$ can be computed in $\mathcal{O}(n\log n)$ computational time using the {\em fast Fourier transform} (FFT).
%Clearly, this is an interpolating approximation as %$\mathcal{T}^d_{n,2}(\hat{f}^c)(x_j) = f(x_j)$ for $j = 0, \ldots, n$. 
%The coefficients $\tilde{c}_k(\hat{f}^c)$ for $k = -(n-1),\ldots,n$, in practice, is obtained in $\mathcal{O}(n\log n)$ computational cost.

%To estimate $\Vert \mathcal{T}^d_{n,2}(\hat{f}^d_c)  - \mathcal{T}_{n,2}(f_c)\Vert_{\infty,[0,1]}$

To guage the accuracy of such approximations, we begin by obtaining an estimate for 
%the difference between discrete and continuous approximations, 
$\Vert \mathcal{T}^d_{n,2}(\hat{f}^d_c)  - \mathcal{T}_{n,2}(f_c)\Vert_{\infty,[0,1]}$. 
%using the fact that $\hat{f}_c(x) = f_c(x)$ 
Toward this, the following straightforward calculation %for $x \in [0,1]$,
\begin{align*}
& \mathcal{T}^d_{n,2}(\hat{f}^d_c)(x) - \mathcal{T}_{n,2}(f_c)(x) 
%= \mathcal{T}^d_{n,2}(\hat{f}^d_c)(x) - \mathcal{T}_{n,2}(f_c)(x) - (\hat{f}_c(x) - f_c(x)) \\ 
=  \sum_{k=-n}^{n-1} \left(c_k^d(\hat{f}^d_c) - c_k(f_c)\right)e^{\pi i k x} - c_{n}(f_c)e^{\pi i n x}  \\
&= \sum_{k=-n}^{n-1} \left(c_k^d(\hat{f}^d_c) - c_k^d(f^d_c)\right)e^{\pi i k x} + \sum_{k=-n}^{n-1} \left(c_k^d(f^d_c) - c_k(f_c)\right)e^{\pi i k x} - c_{n}(f_c)e^{\pi i n x}  \\
&= \sum_{k=-n}^{n-1} e^{\pi i k x} \frac{1}{2n} \sum_{j=-n}^{n-1} \left((\hat{f}^d_c)_j-(f^d_c)_{j}\right) e^{-\pi i j k /n} \\ &+ \sum_{k=-n}^{n-1} \left( \frac{1}{2n}\sum_{j=-n}^{n-1} (f^d_c)_{j} e^{-\pi i j k /n} - c_k(f_c) \right)e^{\pi i k x} - c_{n}(f_c)e^{\pi i n x} \\
& = \sum_{k=-n}^{n-1} e^{\pi i k x} \frac{1}{2n}\sum_{j=-n}^{n-1} \sum_{\ell=-\infty}^{\infty} c_{\ell}(\hat{f}_c-f_c) e^{\pi i j (\ell-k) /n} \\ &+ \sum_{k=-n}^{n-1} \left(\frac{1}{2n}\sum_{j=-n}^{n-1} \sum_{\ell=-\infty}^{\infty} c_{\ell}(f_c) e^{\pi i j (\ell-k) /n} - c_k(f_c)\right)e^{\pi i k x} - c_{n}(f_c)e^{\pi i n x} \\
%& = \sum_{k=-n}^{n-1} e^{\pi i k x} \frac{1}{2n} \sum_{\ell=-\infty}^{\infty}  c_{\ell}(\hat{f}_c-f_c) \sum_{j=-n}^{n-1} e^{\pi i j (\ell-k) /n} \\ &+ \sum_{k=-n}^{n-1} \left(\frac{1}{2n} \sum_{\ell=-\infty}^{\infty} c_{\ell}(f_c) \sum_{j=-n}^{n-1} e^{\pi i j (\ell-k) /n} - c_k(f_c)\right)e^{\pi i k x} - c_{n}(f_c)e^{\pi i n x} \\
%& = \sum_{k=-n}^{n-1} \sum_{\ell=-\infty}^{\infty}  c_{k+2\ell n}(\hat{f}_c-f_c) e^{\pi i (k+2\ell n) x} \\ &+ \sum_{k=-n}^{n-1} \left( \sum_{\ell=-\infty}^{\infty} c_{k+2\ell n}(f_c) - c_k(f_c)\right)e^{\pi i k x} - c_{n}(f_c)e^{\pi i n x} \\
& = \sum_{k=-n}^{n-1} \sum_{\ell=-\infty}^{\infty}  c_{k+2\ell n}(\hat{f}_c-f_c) e^{\pi i (k+2\ell n) x} + \sum_{k=-n}^{n-1}  \sum_{\substack{\ell=-\infty \\ \ell\ne 0}}^{\infty} c_{k+2\ell n}(f_c) e^{\pi i k x} - c_{n}(f_c)e^{\pi i n x},
%\\
%&\le \sum_{k=-n}^{n-1} \frac{1}{2n}\sum_{j=-n}^{-1} \left| \sum_{m=0}^r f_{0m} P_m^0(x) + \sum_{m=0}^r f_{1m} P_m^1(x)\right| + \sum_{|\ell| \ge n} |c_{\ell}(f_c)| \\
%&\le \sum_{k=-n}^{n-1} \frac{1}{2n}\sum_{j=-n}^{-1} \left|(\hat{f}^d_c)_j-(f^d_c)_{j}\right| + \mathcal{O}\left( \frac{1}{n^{r+1}} \right)
\end{align*}
%Using 
and the fact that $\hat{f}_c(x) = f_c(x)$ for $x \in [0,1]$ %, we see that
yields
%\[
%- \sum_{\ell=-\infty}^{\infty} c_{\ell}(\hat{f}_c - f_c)e^{\pi i n x}
%\]
\begin{align*}
\Vert \mathcal{T}^d_{n,2}(\hat{f}^d_c)  - \mathcal{T}_{n,2}(f_c)\Vert_{\infty,[0,1]} &\le 2\sum_{|\ell| \ge n} |c_{\ell}(\hat{f}_c-f_c)| + \sum_{|\ell| \ge n} |c_{\ell}(f_c)|.
\end{align*}
Thus, we have
\begin{align*}
	\Vert \mathcal{T}^d_{n,2}(\hat{f}^d_c)  - \mathcal{T}_{n,2}(f_c)\Vert_{\infty,[0,1]} &\le  \mathcal{O}\left( \frac{1}{n} \right)\Vert F_{n,p}-F \Vert_{max} + \mathcal{O}\left( \frac{1}{n^{r+1}} \right),
\end{align*}
where the second term in the last inequality follows from 
%\cref{lemma-thm}
the fact that Fourier coefficients $c_k(g)$ decay as $|k|^{-(r+2)}$ for $g \in D_0^{r,1}([-1,1])$ 
while the first term results from the fact that, for $k \ne 0$, we have
\begin{align*}
&c_k(\hat{f}_c-f_c) = \frac{1}{2}\int_{-1}^0 \mathcal{P}_r(F_{n,p}-F)(x)e^{-\pi i k x}\,dx = \\ &\frac{1}{2\pi^2 k^2}\left( \sum_{m=0}^r (f_{0,m}-f^{(m)}(0)) (P_m^0)'(0) + \sum_{m=0}^r (f_{1,m}-f^{(m)}(1)) (P_m^1)'(0)\right) - \\ 
& \frac{(-1)^k}{2\pi^2 k^2}\left( \sum_{m=0}^r (f_{0,m}-f^{(m)}(0)) (P_m^0)'(-1) + \sum_{m=0}^r (f_{1,m}-f^{(m)}(1)) (P_m^1)'(-1)\right) - \\
&\frac{1}{2\pi^2 k^2}\int_{-1}^{0} \left( \sum_{m=0}^r (f_{0,m}-f^{(m)}(0)) (P_m^0)''(x) + \sum_{m=0}^r (f_{1,m}-f^{(m)}(1)) (P_m^1)''(x)\right) e^{-\pi i k x}\,dx.
\end{align*}
Finally, using the inequality
\[
\Vert \mathcal{T}^d_{n,2}(\hat{f}^d_c) - f \Vert_{\infty,[0,1]} \le \Vert \mathcal{T}^d_{n,2}(\hat{f}^d_c) - \mathcal{T}_{n,2}(f_c) \Vert_{\infty,[0,1]} + \Vert \mathcal{T}_{n,2}(f_c) - f \Vert_{\infty,[0,1]}
\]
in conjunction of \cref{thm}, we see that
\[
\Vert \mathcal{T}^d_{n,2}(\hat{f}^d_c) - f \Vert_{\infty,[0,1]} \le \mathcal{O}\left( \frac{1}{n^{p+1}} \right) + \mathcal{O}\left( \frac{1}{n^{r+1}} \right) = \mathcal{O}\left( \frac{1}{n^{\min\{p,r\}+1}} \right)
\]

\begin{table}[t]
	\begin{tabularx}{0.99\textwidth}{ | >{\setlength\hsize{0.4\hsize}\centering}X | >{\setlength\hsize{1.3\hsize}\centering}X | >{\setlength\hsize{0.9\hsize}\centering}X | >{\setlength\hsize{1.3\hsize}\centering}X | >{\setlength\hsize{0.9\hsize}\centering}X | >{\setlength\hsize{1.3\hsize}\centering}X |
			>{\setlength\hsize{0.9\hsize}\centering}X | }
		\hline
		\multirow{2}{\hsize}{$n$} &  \multicolumn{2}{>{\setlength\hsize{2.2\hsize}\centering}X |}{$r = 1$} &
		\multicolumn{2}{>{\setlength\hsize{2.2\hsize}\centering}X |}{$r = 2$} &
		\multicolumn{2}{>{\setlength\hsize{2.2\hsize}\centering}X |}{$r = 3$} \tabularnewline
		%\hline
		\cline{2-7}
		& $e_n$ & $e_n/e_{n-1}$ &  $e_n$  & $e_n/e_{n-1}$  & $e_n$ & $e_n/e_{n-1}$ \tabularnewline
		\hline
		$2^{6}$ & $1.17\times 10^{-3}$ & --- & $2.39\times 10^{-4}$ & --- & $1.93\times 10^{-4}$ & --- \tabularnewline
		\hline
		$2^{7}$ & $3.20\times 10^{-4}$ & $3.65$ & $1.90\times 10^{-5}$ & $12.54$ & $1.24\times 10^{-5}$ & $15.59$ \tabularnewline
		\hline
		$2^{8}$ & $8.24\times 10^{-5}$ & $3.88$ & $1.68\times 10^{-6}$ & $11.32$ & $7.85\times 10^{-7}$ & $15.78$ \tabularnewline
		\hline
		$2^{9}$ & $2.07\times 10^{-5}$ & $3.98$ & $1.73\times 10^{-7}$ & $9.73$ & $4.93\times 10^{-8}$ & $15.94$ \tabularnewline
		\hline
		$2^{10}$ & $5.20\times 10^{-6}$ & $3.99$ & $2.15\times 10^{-8}$ & $8.04$ & $3.09\times 10^{-9}$ & $15.93$ \tabularnewline
		\hline
		$2^{11}$ & $1.22\times 10^{-6}$ & $4.27$ & $2.69\times 10^{-9}$ & $7.99$ & $1.86\times 10^{-10}$ & $16.64$ \tabularnewline
		\hline
		$2^{12}$ & $2.92\times 10^{-7}$ & $4.17$ & $3.36\times 10^{-10}$ & $8.10$ & $1.16\times 10^{-11}$ & $15.98$ \tabularnewline
		%& & & & & \tabularnewline
		\hline
	\end{tabularx}
	\caption{Convergence study for approximations of $f(x) = \sin(20x)$ using the derivative approximations of order $p = 3$ and the continuation polynomial of degree $2r+1$ for $r = 1, 2$ and $3$.}
	\label{sin_p_3}
\end{table}

\subsection{Numerical examples}

We now discuss some numerical experiments to demonstrate that theoretical convergence rates obtained above are indeed achieved in practice. Toward this, we consider the problem of approximating a function $f(x)$ on $[0,1]$ using the functional data on a uniform grid of size $n$. We record the relative approximation error $e_n$ that is obtained as
\[
e_n = \max_{0 \le j \le N} | \mathcal{T}_{n,2}^d(\hat{f}_c)(z_j)-f(z_j)|/\max_{0 \le j \le N} | f(z_j)| 
\]
where $N = 2^{13}$ and $z_j = j/N$ are the evaluations points on a large uniform grid where approximate and exact values are compared.

In the first set of experiments, we study the effect of $p$ and $r$ on the rate of convergence as $n$ increases. The results in \cref{sin_p_3} and \cref{sin_p_4} for a smooth function $f(x) = \sin(20x)$ clearly show that the numerical rate of convergence indeed matches the theoretical rate $\min\{p,r\}+1$. Moreover, as expected, the quality of approximations remains satisfactory even for highly oscillatory functions, as seen in \cref{ec_r_4_p_4}.

\begin{table} [!t]
	\begin{tabularx}{0.99\textwidth}{ | >{\setlength\hsize{0.4\hsize}\centering}X | >{\setlength\hsize{1.3\hsize}\centering}X | >{\setlength\hsize{0.9\hsize}\centering}X | >{\setlength\hsize{1.3\hsize}\centering}X | >{\setlength\hsize{0.9\hsize}\centering}X | >{\setlength\hsize{1.3\hsize}\centering}X |
			>{\setlength\hsize{0.9\hsize}\centering}X | }
		\hline
		\multirow{2}{\hsize}{$n$} &  \multicolumn{2}{>{\setlength\hsize{2.2\hsize}\centering}X |}{$r = 2$} &
		\multicolumn{2}{>{\setlength\hsize{2.2\hsize}\centering}X |}{$r = 3$} &
		\multicolumn{2}{>{\setlength\hsize{2.2\hsize}\centering}X |}{$r = 4$} \tabularnewline
		%\hline
		\cline{2-7}
		& $e_n$ & $e_n/e_{n-1}$ &  $e_n$  & $e_n/e_{n-1}$  & $e_n$ & $e_n/e_{n-1}$ \tabularnewline
		\hline
		$2^{6}$ & $1.42\times 10^{-4}$ & --- & $6.94\times 10^{-5}$ & --- & $4.03\times 10^{-5}$ & --- \tabularnewline
		\hline
		$2^{7}$ & $1.28\times 10^{-5}$ & $11.10$ & $2.53\times 10^{-6}$ & $27.39$ & $1.42\times 10^{-6}$ & $28.32$ \tabularnewline
		\hline
		$2^{8}$ & $1.44\times 10^{-6}$ & $8.85$ & $1.02\times 10^{-7}$ & $24.81$ & $4.59\times 10^{-8}$ & $31.04$ \tabularnewline
		\hline
		$2^{9}$ & $1.75\times 10^{-7}$ & $8.23$ & $4.64\times 10^{-9}$ & $22.02$ & $1.44\times 10^{-9}$ & $31.84$ \tabularnewline
		\hline
		$2^{10}$ & $2.16\times 10^{-8}$ & $8.13$ & $2.32\times 10^{-10}$ & $19.96$ & $4.51\times 10^{-11}$ & $31.94$ \tabularnewline
		\hline
		$2^{11}$ & $2.69\times 10^{-9}$ & $8.01$ & $1.27\times 10^{-11}$ & $18.28$ & $1.32\times 10^{-12}$ & $34.16$ \tabularnewline
		\hline
		$2^{12}$ & $3.37\times 10^{-10}$ & $8.00$ & $7.46\times 10^{-13}$ & $17.04$ & $7.67\times 10^{-14}$ & $17.21$ \tabularnewline
		%& & & & & \tabularnewline
		\hline
	\end{tabularx}
	\caption{Convergence study for approximations of $f(x) = \sin(20x)$ using the derivative approximations of order $p = 4$ and the continuation polynomial of degree $2r+1$ for $r = 2, 3$ and $4$.}
	\label{sin_p_4}
\end{table}

\begin{table} [!b]
	\begin{tabularx}{0.99\textwidth}{ | >{\setlength\hsize{0.4\hsize}\centering}X | >{\setlength\hsize{1.3\hsize}\centering}X | >{\setlength\hsize{0.9\hsize}\centering}X | >{\setlength\hsize{1.3\hsize}\centering}X | >{\setlength\hsize{0.9\hsize}\centering}X | >{\setlength\hsize{1.3\hsize}\centering}X |
			>{\setlength\hsize{0.9\hsize}\centering}X | }
		\hline
		\multirow{2}{\hsize}{$n$} &  \multicolumn{2}{>{\setlength\hsize{2.2\hsize}\centering}X |}{$k = 50$} &
		\multicolumn{2}{>{\setlength\hsize{2.2\hsize}\centering}X |}{$k = 100$} &
		\multicolumn{2}{>{\setlength\hsize{2.2\hsize}\centering}X |}{$k = 200$} \tabularnewline
		%\hline
		\cline{2-7}
		& $e_n$ & $e_n/e_{n-1}$ &  $e_n$  & $e_n/e_{n-1}$  & $e_n$ & $e_n/e_{n-1}$ \tabularnewline
		\hline
		$2^{6}$ & $6.78\times 10^{-2}$ & --- & $4.55\times 10^{-1}$ & --- & $1.15\times 10^{-0}$ & --- \tabularnewline
		\hline
		$2^{7}$ & $1.02\times 10^{-3}$ & $66.28$ & $6.61\times 10^{-2}$ & $6.89$ & $4.73\times 10^{-1}$ & $2.43$ \tabularnewline
		\hline
		$2^{8}$ & $3.03\times 10^{-6}$ & $337$ & $1.32\times 10^{-3}$ & $50.07$ & $1.07\times 10^{-1}$ & $4.44$ \tabularnewline
		\hline
		$2^{9}$ & $7.21\times 10^{-8}$ & $42.08$ & $6.87\times 10^{-6}$ & $192$ & $3.94\times 10^{-3}$ & $27.07$ \tabularnewline
		\hline
		$2^{10}$ & $1.95\times 10^{-9}$ & $36.95$ & $1.42\times 10^{-7}$ & $48.44$ & $8.10\times 10^{-6}$ & $486$ \tabularnewline
		\hline
		$2^{11}$ & $5.45\times 10^{-11}$ & $35.79$ & $3.86\times 10^{-9}$ & $36.74$ & $1.12\times 10^{-7}$ & $72.48$ \tabularnewline
		\hline
		$2^{12}$ & $1.63\times 10^{-12}$ & $33.41$ & $1.14\times 10^{-10}$ & $33.86$ & $3.50\times 10^{-9}$ & $31.93$ \tabularnewline
		%& & & & & \tabularnewline
		\hline
	\end{tabularx}
	\caption{Convergence study for approximations of $\exp(-2\cos k x)$ with $k = 50, 100$ and $200$ using the derivative approximations of order $p = 4$ and the continuation polynomial of degree $9$.}
	\label{ec_r_4_p_4}
\end{table}

\begin{table} [!t]
	\begin{tabularx}{0.99\textwidth}{ | >{\setlength\hsize{0.4\hsize}\centering}X | >{\setlength\hsize{1.3\hsize}\centering}X | >{\setlength\hsize{0.9\hsize}\centering}X | >{\setlength\hsize{1.3\hsize}\centering}X | >{\setlength\hsize{0.9\hsize}\centering}X | >{\setlength\hsize{1.3\hsize}\centering}X |
			>{\setlength\hsize{0.9\hsize}\centering}X | }
		\hline
		\multirow{2}{\hsize}{$n$} &  \multicolumn{2}{>{\setlength\hsize{2.2\hsize}\centering}X |}{$p = 1$} &
		\multicolumn{2}{>{\setlength\hsize{2.2\hsize}\centering}X |}{$p = 2$} &
		\multicolumn{2}{>{\setlength\hsize{2.2\hsize}\centering}X |}{$p = 3$} \tabularnewline
		%\hline
		\cline{2-7}
		& $e_n$ & $e_n/e_{n-1}$ &  $e_n$  & $e_n/e_{n-1}$  & $e_n$ & $e_n/e_{n-1}$ \tabularnewline
		\hline
		$2^{6}$ & $1.54\times 10^{-4}$ & --- & $3.20\times 10^{-6}$ & --- & $3.29\times 10^{-6}$ & --- \tabularnewline
		\hline
		$2^{7}$ & $3.88\times 10^{-5}$ & $3.97$ & $4.02\times 10^{-7}$ & $7.94$ & $4.18\times 10^{-7}$ & $7.88$ \tabularnewline
		\hline
		$2^{8}$ & $9.74\times 10^{-6}$ & $3.98$ & $5.05\times 10^{-8}$ & $7.97$ & $5.26\times 10^{-8}$ & $7.94$ \tabularnewline
		\hline
		$2^{9}$ & $2.43\times 10^{-6}$ & $4.01$ & $6.32\times 10^{-9}$ & $7.98$ & $6.59\times 10^{-9}$ & $7.98$ \tabularnewline
		\hline
		$2^{10}$ & $6.08\times 10^{-7}$ & $4.00$ & $7.81\times 10^{-10}$ & $8.10$ & $8.17\times 10^{-10}$ & $8.06$ \tabularnewline
		\hline
		$2^{11}$ & $1.46\times 10^{-7}$ & $4.16$ & $9.77\times 10^{-11}$ & $8.00$ & $1.02\times 10^{-10}$ & $7.99$ \tabularnewline
		\hline
		$2^{12}$ & $3.65\times 10^{-8}$ & $4.00$ & $1.22\times 10^{-11}$ & $8.00$ & $1.28\times 10^{-11}$ & $8.00$ \tabularnewline
		%& & & & & \tabularnewline
		\hline
	\end{tabularx}
	\caption{Convergence study for approximations of $f(x) = |x-1/3|(x-1/3)^2$ using the continuation polynomial of degree $5$ and the derivative approximations of order $p = 1, 2$ and $3$.}
	\label{sin_r_2}
\end{table}

\begin{table} [!b]
	\begin{tabularx}{0.99\textwidth}{ | >{\setlength\hsize{0.4\hsize}\centering}X | >{\setlength\hsize{1.3\hsize}\centering}X | >{\setlength\hsize{0.9\hsize}\centering}X | >{\setlength\hsize{1.3\hsize}\centering}X | >{\setlength\hsize{0.9\hsize}\centering}X | >{\setlength\hsize{1.3\hsize}\centering}X |
			>{\setlength\hsize{0.9\hsize}\centering}X | }
		\hline
		\multirow{2}{\hsize}{$n$} &  \multicolumn{2}{>{\setlength\hsize{2.2\hsize}\centering}X |}{$p = 1$} &
		\multicolumn{2}{>{\setlength\hsize{2.2\hsize}\centering}X |}{$p = 2$} &
		\multicolumn{2}{>{\setlength\hsize{2.2\hsize}\centering}X |}{$p = 3$} \tabularnewline
		%\hline
		\cline{2-7}
		& $e_n$ & $e_n/e_{n-1}$ &  $e_n$  & $e_n/e_{n-1}$  & $e_n$ & $e_n/e_{n-1}$ \tabularnewline
		\hline
		$2^{6}$ & $1.56\times 10^{-4}$ & --- & $2.60\times 10^{-6}$ & --- & $8.13\times 10^{-7}$ & --- \tabularnewline
		\hline
		$2^{7}$ & $3.91\times 10^{-5}$ & $4.00$ & $3.15\times 10^{-7}$ & $8.25$ & $1.01\times 10^{-7}$ & $8.03$ \tabularnewline
		\hline
		$2^{8}$ & $9.79\times 10^{-6}$ & $4.00$ & $3.88\times 10^{-8}$ & $8.12$ & $1.27\times 10^{-8}$ & $8.01$ \tabularnewline
		\hline
		$2^{9}$ & $2.44\times 10^{-6}$ & $4.02$ & $4.79\times 10^{-9}$ & $8.09$ & $1.58\times 10^{-9}$ & $8.00$ \tabularnewline
		\hline
		$2^{10}$ & $6.09\times 10^{-7}$ & $4.00$ & $5.97\times 10^{-10}$ & $8.03$ & $1.98\times 10^{-10}$ & $8.00$ \tabularnewline
		\hline
		$2^{11}$ & $1.46\times 10^{-7}$ & $4.17$ & $7.15\times 10^{-11}$ & $8.35$ & $2.27\times 10^{-11}$ & $8.72$ \tabularnewline
		\hline
		$2^{12}$ & $3.66\times 10^{-8}$ & $4.00$ & $8.93\times 10^{-12}$ & $8.01$ & $2.84\times 10^{-12}$ & $7.99$ \tabularnewline
		%& & & & & \tabularnewline
		\hline
	\end{tabularx}
	\caption{Convergence study for approximations of $f(x) = |x-1/3|(x-1/3)^2$ using the continuation polynomial of degree $7$ and the derivative approximations of order $p = 1, 2$ and $3$.}
	\label{sin_r_3}
\end{table}

\begin{table} [!t]
	\begin{tabularx}{0.99\textwidth}{ | >{\setlength\hsize{0.4\hsize}\centering}X | >{\setlength\hsize{1.27\hsize}\centering}X | >{\setlength\hsize{0.875\hsize}\centering}X | >{\setlength\hsize{1.27\hsize}\centering}X | >{\setlength\hsize{0.875\hsize}\centering}X | >{\setlength\hsize{1.27\hsize}\centering}X |
	>{\setlength\hsize{1.02\hsize}\centering}X | }
		\hline
		\multirow{2}{\hsize}{$n$} &  \multicolumn{2}{>{\setlength\hsize{2.2\hsize}\centering}X |}{$\epsilon = 1$} &
		\multicolumn{2}{>{\setlength\hsize{2.2\hsize}\centering}X |}{$\epsilon = 0.1$} &
		\multicolumn{2}{>{\setlength\hsize{2.2\hsize}\centering}X |}{$\epsilon = 0.01$} \tabularnewline
		%\hline
		\cline{2-7}
		& $e_n$ & $e_n/e_{n-1}$ &  $e_n$  & $e_n/e_{n-1}$  & $e_n$ & $e_n/e_{n-1}$ \tabularnewline
		\hline
		$2^{6}$ & $1.43\times 10^{-9}$ & --- & $1.39\times 10^{-7}$ & --- & $2.06\times 10^{-1}$ & --- \tabularnewline
		\hline
		$2^{7}$ & $4.24\times 10^{-11}$ & $33.67$ & $4.07\times 10^{-9}$ & $34.12$ & $3.02\times 10^{-2}$ & $6.81$ \tabularnewline
		\hline
		$2^{8}$ & $1.29\times 10^{-12}$ & $32.81$ & $1.21\times 10^{-10}$ & $33.53$ & $5.98\times 10^{-4}$ & $50.52$ \tabularnewline
		\hline
		$2^{9}$ & $3.99\times 10^{-14}$ & $32.43$ & $3.68\times 10^{-12}$ & $32.99$ & $1.92\times 10^{-7}$ & $3123$ \tabularnewline
		\hline
		$2^{10}$ & $9.55\times 10^{-15}$ & $4.17$ & $1.11\times 10^{-13}$ & $33.07$ & $2.97\times 10^{-14}$ & $6.46\times 10^6$ \tabularnewline
		\hline
		%$2^{11}$ & $1.23\times 10^{-14}$ & $0.78$ & $3.73\times 10^{-15}$ & $29.8$ & $1.37\times 10^{-14}$ & $2.16$ \tabularnewline
		%\hline
		%$2^{12}$ & $1.68\times 10^{-14}$ & $0.74$ & $3.94\times 10^{-15}$ & $0.95$ & $1.47\times 10^{-14}$ & $0.93$ \tabularnewline
		%& & & & & \tabularnewline
		%\hline
	\end{tabularx}
	\caption{Convergence study for approximations of $f_{\epsilon}(x) = ((x-1/3)^2 + \epsilon^2)^{-1}$ with $\epsilon = 1, 0.1$ and $0.01$ using the derivative approximations of order $p = 4$ and the continuation polynomial of degree $9$.}
	\label{runge_r_4_p_4}
\end{table}

Next, we take the function $f(x) = |x-1/3|(x-1/3)^2 \in D^{2,1}([0,1])$, where the convergence rate increases as the order of derivative approximations improves, but only up to cubic convergence, as seen in \cref{sin_r_2}. The results in \cref{sin_r_3} confirm that, unlike the previous smooth cases, increasing the value of $r$ in the Fourier continuation approximation beyond $2$ does not bring additional gains in terms of convergence speed. 

Finally, we conclude this section by looking at aproximation quality of the proposed approach for $f_{\epsilon}(x) = ((x-1/3)^2+\epsilon^2)^{-1}$ on $[0,1]$
%and $g_{\epsilon}(x) = ((x-1/2)^2-(\epsilon+1/2)^2)^{-1}$ that have 
that has poles in the complex plane at $z = 1/3 \pm \epsilon i$. 
%and $-\epsilon, 1+\epsilon$ respectively. The results in \
The Fourier continuation approximations used for the results in \cref{runge_r_4_p_4} correspond to the parameters $r = 4$, $p = 4$ and, therefore, are expected to converge with rate $5$ as seen in the table, particularly for $\epsilon = 1$ and $\epsilon = 0.1$. It is interesting to note that the results corresponding to $\epsilon = 0.01$ exhibit superalgebraic convergence due to the relativly ``small" boundary data compared to the peak function value $\epsilon^2$ at $x = 1/3$.

\section{Concluding remarks}

In this paper, we analyzed a Fourier approximation strategy for non-periodic functions that, to avoid Gibbs oscillations, utilizes a construction for their smooth continuation to a larger interval so that the continued function is periodic. We were able to show that such approximations indeed converge with high-order. In particular, we investigated the two-point Hermite polynomial based continuation strategy and found that they are not only simple to implement but also high-order accurate. Further, in the discrete setting where functional data is available only on an equispaced grid, this construction was utilized to obtain interpolatory trigonometric approximations that converge with high-order and has $\mathcal{O}(n\log n)$ computational complexity. Our numerical experiments validate the performance of this scheme in terms of approximation quality and that the theoretical convergence rates are attained in practice. 

While this work focussed mainly on investigating the approximation properties of Fourier continuation strategy, a future step in this direction of significant interest would be to analyze its use in PDE solvers and study the corresponding convergence rates.  

\section*{Acknowledgments}
The author gratefully acknowledges support from IITK-ISRO Space Technology Cell through contract No.  STC/MATH/2014100.

%\bibliographystyle{siamplain}
%\bibliography{../../../../../GoogleDrive/Research/References/referencesbibdesk}

\end{document}